\documentclass[12pt,reqno]{amsart}
\usepackage[T1]{fontenc}
\usepackage[utf8]{inputenc}
\usepackage[english]{babel}
\usepackage{amsmath}
\usepackage{amsfonts}
\usepackage{amssymb}
\usepackage{amsthm}
\usepackage{bm}
\usepackage{braket}
\usepackage{graphicx}
\usepackage{marginnote}
\usepackage[margin=1.5in]{geometry}
\usepackage[colorlinks=true, pdfstartview=FitV, linkcolor=darkblue, citecolor=darkred, urlcolor=cyan]{hyperref}
\usepackage{dsfont} 
\usepackage{pxfonts}

\usepackage{microtype}
\usepackage{color}
\usepackage{csquotes}

\definecolor{darkred}{RGB}{203,65,84}
\definecolor{darkblue}{RGB}{70,130,180}
\definecolor{brown}{RGB}{139,69,19}

\newtheorem{theorem}{Theorem}[subsection]
\newtheorem{lemma}[theorem]{Lemma}
\newtheorem{proposition}[theorem]{Proposition}
\newtheorem{corollary}[theorem]{Corollary}

\newtheorem{definition}[theorem]{Definition}

\newcommand{\T}{\ms T}
\renewcommand{\leq}{\leqslant}

\renewcommand{\epsilon}{\varepsilon}
\newcommand{\mG}{\ms G}
\newcommand{\mH}{\ms H}

\newcommand{\psld}{\mathsf{PSL}_2(\mathbb R)}
\newcommand{\hh}{{\bf H}^2}

\newcommand{\ms}{\mathsf}

\newcommand{\mc}{\mathcal}
\newcommand{\id}{\operatorname{Id}}

\newcommand{\defeq}{\coloneqq}
\newcommand{\eqdef}{\eqqcolon}

\renewcommand{\d}{{\rm d}}

\newcommand{\bX}{\partial_\infty X}

\newcommand{\mK}{\mathsf K}

\newcommand{\im}{\operatorname{Im}}
\newcommand{\Hom}{\operatorname{Hom}}
\newcommand{\Iso}{\operatorname{Iso}}

\newcommand{\bP}{{\mathbf P}}
\newcommand{\cG}{\mathcal G}

\newcommand{\Zreg}{Z_{\tiny{reg}}}

\newcommand{\uu}{\mathsf u}
\newcommand{\mU}{\mathsf U}

\newcommand{\seq}[1]{ \{{#1}_p\}_{p\in\mathbb N}}
\newcommand{\seqp}[1]{ \{{#1}\}_{p\in\mathbb N}}
\newcommand{\seqt}[1]{ \{{#1}\}_{t\in\mathbb R}}

\newcounter{foo}
\newtheorem{theo}[foo]{Theorem}

\title{The topological  Ax--Lindemann theorem\\ for the complex hyperbolic space} 

\author[F. Labourie]{Fran\c cois Labourie}
\address{EPF Lausanne, SB-SCI-FL, Station 8, CH-1015 Lausanne,  Switzerland}
\email{francois.labourie@epfl.ch}
\author[S. Mozes]{Shahar Mozes}
\address{The Hebrew University of Jerusalem, Einstein Institute of Mathematics, Givat Ram. Jerusalem, 9190401, Israel}
\email{mozes@math.huji.ac.il}
\thanks{This material is based upon work supported by the National Science Foundation under Grant No. DMS-1928930, while the two authors were in residence at the Simons Laufer Mathematical Sciences Institute in Berkeley, California, during the first semester of 2026. F.~L. is supported by the Swiss State Secretariat for Education, Research and Innovation (SERI): MB25.00031.  }

\begin{document}
\renewcommand{\theenumi}{(\roman{enumi})}%
\maketitle
\begin{abstract}
	In the case of the complex hyperbolic space, we establish a version of the hyperbolic Ax—Lindemann Theorem for the topological closure rather than the Zariski closure.  Our approach combines classical differential geometry with homogeneous dynamics.
\end{abstract}
\section*{Introduction}
A {\em bounded symmetric domain} $X$ is a bounded open convex set in $\mathbb C^n$ whose group $\mG$ of biholomorphisms acts transitively. Bounded symmetric domains are classified: they are symmetric spaces of non compact types isometric to $\mG/\mK$ where $\mG$ is semisimple, and whose maximal compact subgroup $\mK$ has a center of positive dimension.

When $\Gamma$ is a torsion free lattice in $\mG$, let us denote by $\pi$  a projection from $X$ to $\Gamma\backslash X$. Our main result is when $\mG$ has rank 1, or equivalently $X$ is the unit ball in $\mathbb C^n$. We show the following result

\begin{theo} Let $\Gamma$ be a uniform lattice in the group of biholomorphisms of the unit ball $X$ in $\mathbb C^n$. Let $Z$ be a compact irreducible complex manifold (with boundary) and $f$ an holomorphic map with values in $\mathbb C^n$, such that 
	\begin{enumerate}
		\item $f$ is an immersion somewhere
		\item $f(\partial Z)\cap \bar X=\emptyset$
		\item $f(Z)\cap X\not=\emptyset$ , 
	\end{enumerate}
	Then denoting $\pi$ the projection from $X$ to $\Gamma\backslash X$, the topological closure of the complex analytic variety  $\pi(f(Z)\cap X)$ is a closed totally geodesic hermitian subspace of $\Gamma\backslash X$.
\end{theo}

The general  hermitian case is of  interest for algebraic geometers and number theorists due to its proximity to the André--Oort problem and is then called the Hyperbolic Ax--Lindenmann Theorem. In this context, several authors have exploited  the fact that \(\Gamma \backslash X\) has the structure of an algebraic manifold and studied the Zariski closure of the projection of an affine algebraic submanifold of \(\mathbb{C}^n\).  Related papers by Bruno Klingler, Ngaiming Mok, Jonathan Pila,  Jacob Tsimerman, Emmanuel Ullmo and Andrei Yafaev address this question and provide a complete solution, as detailed in \cite{Ullmo:2014aa,Klingler:2016aa,Mok:2019aa}.

We underline that in our result the algebraic structure on the quotient plays  no role and our result is stronger in the complex hyperbolic space:  we take the topological closure of the projection, and we do not assume any algebraicity on $Z$ nor $f$. More precisely, by Heisuke Hironaka's results \cite{Hironaka:1964aa,Hironaka:1964ab}, every algebraic variety over $\mathbb C$ admits a resolution of singularities. It follows that our result recovers the original Ax--Lindemann theorem in complex hyperbolic space, for uniform lattices, about the Zariski closure of the projection of an algebraic variety, since the Zariski closure contains the topological closure and closed totally geodesic hermitian spaces are algebraic.

However the result that we obtain for complex analytic maps cannot be true in the general hermitian case. Here is a simple counterexample: we take $X=\hh\times \hh$ to be the polydisk in $\mathbb C^2$. We now take the (non-irreducible) $\Gamma=\Gamma_1\times\Gamma_2$ in the isometry group of $X$ containing  $\psld\times\psld$. Then we take a holomorphic map $f$ from the disk of radius $2$ in $\mathbb C$ to the disk of radius 1 whose image is included in a small ball in the fundamental domain of $\Gamma_2$, thus projecting to a strict compact subset $K_2$ of $S_2=X/\Gamma_2$. We finally take $Y$ to be the graph of $f$. Then $\partial Y$ is disjoint from the closure of $X$ and $\pi(Y)$ is not dense in $\Gamma\backslash X$ since it is included in $(X/\Gamma_1\times K_1)$. However it could be  that apart from this trivial counterexample our topological techniques can be applied.

We thank Emmanuel Ullmo for bringing the problem to our attention, as well as  Bruno Klingler, Fanny Kassel, Hee Oh, Peter Sarnak, Jonathan Pila Nimish Shah, and  Mathew Stover for their interest in our work.

\tableofcontents
\section{The complex hyperbolic space and its models}

Recall the construction of projective model of the complex hyperbolic space: let $E$ be a complex vector space of dimension $n+1$ equipped with an hermitian metric of signature $(1,n)$. 

The complex hyperbolic space is the space $X$  of lines in $E$ on which $q$ is positive definite. We see $X$ as an open set in $\mathbf P(E)$. The group of $\mathsf G=\mathsf{PU}(q)$ acts by biholomorphisms on $X$. Conversely every biholomorphism of $X$ belongs to $\mG$.

Let us denote by $\partial_\infty X$ the set of complex lines on which the restriction of $q$ is zero. The set $\partial_\infty X$ is a smooth submanifold. At a point $x$, the maximal complex subspace of $\mathsf T_x\partial_\infty X$ is the orthogonal $x^\perp$ of $x$ with respect to $q$.

Observe also that every vector space passing through a space like line is non degenerate and thus corresponds to a totally geodesic subspace of $X$.

From this description, $X$ comes equipped with a canonical hermitian metric. Indeed we have the identification 
$$
\mathsf T_xX=\Hom(x, x^\perp)\ ,
$$ 
where on the left-hand side we see $x$ as a space like line in $E$ and $x^\perp$ is is orthogonal with respect to $q$. Since the restriction of $q$ to both  $x$ and $x^\perp$ is definite,  $q$ gives rise to  a definite hermitian  metric on $\mathsf T_xX$, which is invariant by $\mathsf G$.

By construction, if $\bP(F)$ is a projective subspace of $\bP(E)$ intersecting $X$, then $\bP(F)\cap X$ is a complex totally geodesic subspace of $X$, and every complex totally geodesic subspace of $X$ is of this form.

Finally observe

\begin{lemma}\label{lem:trans} Let $m$ be an integer.
The group $\mathsf G$ acts transitively on the space  $(L,x)$ of pairs where $L$ is complex projective space of dimension $m$  and $x$ is a point in $L\cap\partial_\infty X$ at which $L$ intersects $\partial_\infty X$ transversely.
\end{lemma}
\begin{proof}
	 Write $L=\bP(F)$; then $(L,x)$  a pair such as in the lemma, if and only if the restriction of $q$ to $F$ has signature $(1,m)$ and $x$ is a lightlike line in $F$. Observe that  $\mG$ acts transitively on the space of $F$ such that $q\vert_F$ has signature $(1,m)$, and the stabiliser of $F$ in $\mG$ acts transitively on the space of lightlike lines of $L$. The result follows.
\end{proof}
\subsection{The unit ball model}\label{sec:unitballmodel}
Let $x$ be point in $X$, let $u$  be a vector in the line  $x$ such that $q(u)=1$. Observe that the restriction of $q$ to $x^\perp$ is negative definite. Let $H$ be a complex  affine hyperplane parallel to $x^\perp$ and containing $u$. We equip $H$ with the hermitian metric $-q$. We  consider the holomorphic affine chart $\psi$ from $H$ to $\bP(E)$ such that $\psi(v)$ is the complex line generated by $v$. Then $X$ is bihomolorphic to the unit ball $B=\psi^{-1}(X)$ in $H$. 

Choosing a hermitian basis in $x^\perp$, hence identifying $H$ with $\mathbb C^n$, we have Thus identified biholomorphically $X$ with the unit ball in $\mathbb C^n$. This construction is the {\em unit ball model of the complex hyperbolic space}.
  
\subsection{A upper half plane model and the unipotent group}

An upper half plane model is a biholomorphism of $X$ with 
$$
\mathcal H\defeq \left\{(z_1,\ldots,z_n)\in\mathbb C^n\ \mid \Re(z_1)+\sum_{j=2}^n z_i\overline{z_j}<0\right\}\ .
$$
Let 
$$
\mathcal P\defeq \left\{(z_1,\ldots,z_n)\in \mathbb C^n\ \mid \Re(z_1)+\sum_{j=2}^n z_i\overline{z_j}=0\right\}\ .
$$
Observe that $\mathcal P$ is a smooth submanifold and that if $u=(u_1,\ldots, u_n)$ belongs to $\mc P$
$$
\T_u\mathcal P=\left\{(v_1,\ldots,v_n)\mid \Re(v_1)+\sum_{j=2}^n \Re(v_j\bar u_j)=0\right\} \ .
$$
In the upper plane model, let us consider the
map defined for $t$ in $\mathbb R$ by 
$$\uu_t\ :\ (z_1,\ldots,z_n)\mapsto\frac{1}{1+itz_1}(z_1,\ldots,z_n)\ .
$$
Then we have
\begin{proposition}\label{pro:unip}
	The transformation $\uu_t$ is well defined on $\mc H$, is a unipotent in $\mG$, preserves $(0,\ldots,0)$ the complex line $L$ defined by the equations $z_2=\ldots=z_n=0$. The set $\mU=\seqt{\uu_t}$ is a one-parameter unipotent subgroup of $\mG$.
	
	Moreover $\uu_t$ preserves the family of euclidean normal to $L$
	$$
	P_{z_1}=\{(z_1,u_2,\ldots,u_n)\mid u_i\in \mathbb C\},
	$$
	for $z_1$ with $\Re(z_1)<0$.
\end{proposition}
\begin{proof}
	For every point $z=(z_1,\ldots,z_n)$ in $\mc H$, $\Re(z_1)<0$. Thus $1+itz_1\not=0$ and $\uu_t(z)$ is well defined.
	Observe now that 
	$$
	\Re\left(\frac{z_1}{1+iz_1t}\right)=\frac{1}{\Vert 1+iz_1t\Vert^2}\Re(z_1)\ .
	$$
	Thus $\uu_t(\mc H)\subset \mc H$. Moreover $\uu_t\circ \uu_{-t}=\id$, thus $\uu_t$ is a biholomorphism of $\mc H$ hence an element of $\mG$.
	
	Let us consider the projective involution $\Psi$, given by 
	$$
\Psi(z_1,\ldots,z_n)=\frac{1}{z_1}\left(1,z_2,\ldots,z_n\right)\ ,
	$$
	Then one checks that 
\begin{align}
\Psi^{-1}\uu_t\Psi(z_1,\ldots,z_n)&=\Psi^{-1}\uu_t\left(\frac{1}{z_1},\frac{z_2}{z_1},\ldots,\frac{z_n}{z_1}\right)
\\&=
\Psi^{-1}\left(\frac{1}{1 +\frac{it}{z_1}}\left(\frac{1}{z_1},\frac{z_2}{z_1},\ldots,\frac{z_n}{z_1}\right)\right)
\\&=
\Psi^{-1}\left(\frac{1}{z_1 +it}\left(1,z_2,\ldots,z_n\right)\right)
\\&=(z_1+it, z_2,\ldots,z_n)\ .\end{align}

After this change of coordinates, one then see that $\mU$ is a one parameter subgroup of unipotent elements.  
	
Obviously $\uu_t$ fixes zero and preserves $L$. Finally the fact the the family of euclidean normal to $L$ comes from the fact that 
	$$
	\uu_t(P_{z_1})= P_{\frac{z_1}{1+iz_1t}}
	$$	
	This concludes the proof.	
\end{proof}
We have to understand groups that contain $\mU$.

\begin{proposition}\label{prop:uug2}
	Let $\mG_2$ be a reductive subgroup of $\mG_0$ containing $\mU$. Then the noncompact part of $\mG_2$ is conjugated to $\mathsf{SU}(1,m)$ for some $m$, moreover $\mG_2$ preserves a unique complex totally geodesic subspace $N$ of dimension $m$ in $X$, and $\partial_\infty N$ contains $0$. 
\end{proposition}
\begin{proof}
The unipotent subgroup $\mU$ is included in a group $\mH_0$ which is conjugated to $\mathsf{SU}(1,1)$ in $\mG$. On the other hand, $\mU$ is included in the noncompact part $\mG_3$ of $\mG_2$. Applying   Jacobson--Morozov Theorem \cite[Proposition 1, Proposition 2, §11] {Bourbaki:owAvyv1m} to the semisimple part of $\mG_2$, implies that $\uu$ is contained in some subgroup $\mH_1$ of $\mG_3$ which is conjugated to $\mathsf{SU}(1,1)$. However since $\mG_3$ is non compact and included in $\mathsf{SU}(1,n)$, it is either conjugated to $\mathsf{SO}(1,m)$ or  $\mathsf{SU}(1,m)$ for some $m$. Groups conjugated to $\mathsf{SO}(1,m)$ cannot contain groups conjugated to  $\mathsf{SU}(1,1)$: indeed the totally geodesic subspace associated to $\mathsf{SO}(1,n)$ is totally real, while the totally geodesic subspace associated to $\mathsf{SU}(1,1)$ is complex. It follows that $\mG_3$ is conjugated to  $\mathsf{SU}(1,m)$  for some $m$. Thus there is a unique complex totally geodesic subspace $N$ preserved by $\mG_2$ and in particular by $\mU$. Let $z$ be in $L$ and $w$ in $N$. Observe that $d(\uu_t(z),\uu_t(w))$  is constant since $\uu_t$ is an isometry for the complex hyperbolic distance, this implies that 
$$
\lim_{t\to\infty}(\uu_t(w))=\lim_{t\to\infty}(\uu_t(z))=0\ .
$$
Hence $0$ belongs to $\partial_\infty N$ which is what we wanted to prove.
\end{proof}

\subsection{The projective space of $X$} For any complex manifold $M$ we consider
{\em projective space} of $M$ as the fiber bundle  over $M$ whose fiber at $x$ is the projective space of $\T_xM$: 
$$
\bP(M)\defeq\{(x,L)\mid L\in \bP(\T_xM)\}.
$$
When $M$ is a complex hyperbolic manifold $Y$, for every $(x,L)$ in $\bP(Y)$ there a unique $1$-dimensional complex totally geodesic hyperbolic space $N$ such that $x$ belong to $N$ and $L=\T_xN$. For any $1$-dimensional complex totally geodesic hyperbolic space $N$ immersed in $Y$, we denote by $h(N)$ the {\em lift of $N$} in $\bP(Y)$ as 
$$
h(N)=\{(y,\T_y N)\mid y\in N\}\ .
$$ 
Let  
$$\mG_0\defeq\mathsf{SU}(1,n)\ ,  \ \mG_1\defeq\mathsf{S}^1 \times\mathsf{SU}(n-1)\subset \mG_0\ , \ \mc G\defeq\Iso(\mG_0,\mG)\ ,$$ where ${\mathsf S}^1$ acts by 
multiplication on the second factor, $\mathsf{SU}(n-1)$ acts on the last coordinates. Then
$$
\bP(X)=\mc G/\mG_1\  .
$$
Since $\mG_1$ normalizes $\mathsf{SU}(1,1)$ --considered as acting on the first two coordinates, the action of $\mathsf{SU}(1,1)$ on  $\mc G$ gives a foliation on $\bP(X)$. The leaf of this foliation are precisely the immersed submanifold $h(N)$. Observe finally that $\mG$ acts on the left on $\cG$ and that for a torsion free discrete subgroup of $\mG$
$$
{\bP}(\Gamma\backslash X)=\Gamma\backslash\bP(X)\ .
$$
\section{Complex manifolds and holomorphic maps}
\subsection{Holomorphic maps}
Let $Z$ be a connected compact complex manifold of positive dimension $m$,  and boundary $\partial Z$. Let  $\phi$ be a holomorphic map from $Z$ with values in ${\bf P}(E)$.

\begin{definition}[\sc Hypothesis $(*)$]\label{sec:hypo*}
	We say $(Z,\phi)$  {\em satisfies hypothesis $(*)$} if  
\begin{enumerate}
	\item the set of  points $x$ in $Z$  where the tangent map  $\T_x\phi$ is injective is non empty (hence open and dense).	\item $\phi(Z)\cap X$ is non empty,
	\item $\phi(\partial Z)$ is a subset of ${\bf P}(E)\setminus X$.
\end{enumerate} 

\end{definition} 
Observe that  hypothesis (i) implies that   $m\leq n$ and that $\phi$ is not constant.  Then  (ii) implies that $Z\cap \partial_\infty X$ is not empty, since $E$ does not contain the non constant image  of any closed complex submanifold of positive dimension. 

Observe also that the hypotheses imply  that $\partial Z$ is not empty.

While we shall restrict ourselves to the case $Z$ is a curve in the sequel, for the moment $Z$ will have any dimension less than $n$.

\subsection{Regular points}

\begin{definition}[\sc Regular points]

We say a point $x$ in $Z$ is {\em regular with respect to $\partial_\infty X$}  if 
\begin{enumerate}
	\item $\phi(x)$ belongs to $\partial_\infty X$ ,
	\item $\T_x\phi$ is injective and $\im(\T_x\phi)$ is transverse to $\bX$ at $x$.
\end{enumerate}
We denote the set of regular points by $\Zreg$
\end{definition}

\subsection{The set of regular points}
We prove the following result
\begin{proposition}\label{pro:reg} The set $\Zreg$ of regular points is open and nonempty in $\phi^{-1}(\partial_\infty X)$.
\end{proposition}

Let $Z^0$ the set of points in $Z$ where $\T\phi$ is injective. Observe that $Z^0$ is open. The set of points $\Zreg^0$ of $ \phi^{-1}(\bX)$ transverse to $\bX$ is obviously open in $Z^0\cap \phi^{-1}(\bX)$. Thus $\Zreg$ is open  in $\phi^{-1}(\bX)$.  The following lemma implies the result.

\begin{lemma}\label{lem:non-empty1}
The set $\Zreg^0$ is non empty.
\end{lemma}
We now consider, as per  section \ref{sec:unitballmodel}, the unit ball model of the complex hyperbolic space in $\mathbb C^n$ the model. Let $r$ with $r\leq 1$ and 
$$K_r\defeq\{x\in Z\ ,\  \Vert \phi(z)\Vert=r\}\ ,\ U_r\defeq \{x\in Z\ ,\  \Vert \phi(z)\Vert \le r\}\ .$$
By construction $K_1=\phi^{-1}(\bX)$

By Sard Theorem, there exists a sequence $\seq{r}$ converging to 1 so that for all $p$,
$U_{r_p}$ is a 	compact real submanifold  with boundary contained in $Z$   satisfying 
$$
\partial U_{r_p}=K_{r_p}\ .
$$
Furthermore, using the same notation as above, 
\begin{lemma}\label{lem:2n-3}
There exists a constant $A$, so that 
$\mu(K_{r_p})\leq A , $	
where  $\mu$ is the $2m-1$ dimensional measure.
\end{lemma}

\begin{proof} This is a statement contained in Paragraph 3 in \cite{Hardt:1983aa}  
\end{proof}

 Let us consider the  forms on $\mathbb C^n$ defined at a point $z$ in $\mathbb C^n$ by 
$$
\omega(u,v)=\Re(\braket{u,iv})\ \ , \ \ \beta_z(u)=\Re(\braket{z,iu})\ .
$$
In the coordinates $(z_1,\ldots,z_n)$ of $\mathbb C^n$, writing $z_j=x_j+iy_j$, with $x_j$ and $y_j$ real, we have 
$$
\omega=\sum_{j=1}^{n}\d x_j\wedge \d y_j\ \  ,\  \ \beta=\frac{1}{2}\left(\sum_{j=1}^{n} x_j \d y_j- y_j\d x_j\right)\ ,
$$
and observe that $\d\beta=\omega$.

Our first lemma is
\begin{lemma}
	There is a point $z$ in $K_1$ at which $\phi^*(\beta\wedge
	\omega^{m-1})\not=0$.
\end{lemma}
\begin{proof} We will work by contradiction and assume that $\phi^*(\beta\wedge\omega^{m-1})=0$ on $K_1$. Then $$
\lim_{p\to\infty}\left(\sup_{z\in K_{r_p}} \Vert \phi^*(\beta\wedge\omega^{m-1})_z\Vert\right)\  =0\ .
$$
Since  
$$
\left\vert\int_{K_{r_p}}\phi^*(\beta\wedge\omega^{m-1})\right\vert \leq \mu(K_{r_p})\ \left(\sup_{z\in K_{r_p}} \Vert \phi^*(\beta\wedge\omega^{m-1})_z\Vert\right)\ ,
$$
we get from lemma \ref{lem:2n-3} that 

$$
\lim_{p\to\infty}\left\vert\int_{K_{r_p}}\phi^*(\beta\wedge\omega^{m-1})\right\vert =0\ .
$$
However, since $U_{r_p}$ is a smooth manifold with boundary $K_{r_p}$ and $\d\beta=\omega$,  Stokes formula gives
$$
\int_{U_{r_p}}\phi^*(\omega^{m})= \int_{K_{r_p}}\phi^*(\beta\wedge\omega^{m-1})\ .
$$
Now observe that the lefthand side is a positive increasing function of $p$, since $\phi$ is an immersion on some open dense set in $Z$ by hypothesis $(*)$. This gives the contradiction.
\end{proof}

Then we get 

\begin{proof}[Proof of Lemma  \ref{lem:non-empty1}] Let $z$ obtained by previous lemma, that is such that 
$$
\phi^*(\beta\wedge\omega^{m-1})_z\not=0\ .
$$
It follows that 
\begin{equation}
\phi^*(\beta)_z\not=0\hbox{ and } 	\phi^*(\omega^{m-1})_z\not=0 \label{eq:n-2}\ .
\end{equation}
From the first assertion, we get that 
$$\im(\T_z\phi)\not\subset \ker(\beta_{\phi(z)})\ .$$
But
$$
\ker(\beta_{\phi(z)})=\{u\mid \Re\left( \braket{u,i \phi(z)}\right)=0\}=i\T_{\phi(z)}\partial_\infty X\ .$$
Thus 
$$\im(\T_z \phi)=i \im(\T_z \phi)\not\subset \T_{f(z)} \ \partial_\infty X \ .$$
Thus $\phi$ is transverse to $\partial_\infty X$ at $z$. Hence  $M\defeq \phi^{-1}(\partial_\infty X)$ is a real hypersurface  of $Z$ near $z$. 

Let $V\defeq \T_z M\cap i \T_z M$. Then $V$ is a complex hyperspace of $\T_z Z$ and we may write the orthogonal decomposition
$$
\T_z M=V\oplus \mathbb R u\ ,
$$
for some $u$. Observe that $\iota_u\omega\big\vert_{\T_zM}=0$ and  since $\phi^*(\omega^{{m-1}})_z$ is non zero, it follows from the second part of assertion \eqref{eq:n-2} that \begin{equation}
	\phi^*(\omega^{m-1})\big\vert_V\not=0\ .\label{eq:n-2(1)}
\end{equation}
Observe now that  $iu$ is not in $T_z M$ and thus $\T_zf (iu)$ is non zero. It follows that 
$$
\phi^*\omega(u,iu)\not=0\ .
$$
Hence combining with assertion \eqref{eq:n-2(1)}
$$
\phi^*(\omega^{m})\not=0\ .
$$
Thus $\phi$ is an immersion at $z$. It follows that $z$ is a regular point.
\end{proof}
\subsection{Regular points and totally geodesic submanifolds}

Let $\mc N$ be the set of strict complex projective subspaces $N$ in of $X$ (meaning, we exclude $X$ itself), such that the projection of $N\cap X$ on $Y=\Gamma\backslash X$ is closed and nonempty.

\begin{proposition}
	The set $\mc N$ is countable.
\end{proposition}
\begin{proof} Indeed, since $\mathsf{SU}(1,d)$ has rank 1, the set $\mc N$ injects in the set of conjugacy classes of finitely generated  subgroups of $\Gamma$, which is a countable set.
\end{proof}

Let $(Z,\phi)$ satisfying hypothesis $(*)$ of definition \ref{sec:hypo*}, then we have
\begin{proposition}\label{pro:regnotin}
	Assume $\phi(Z)$ is not included in a complex projective space $M$ whose projection on $\Gamma\backslash X$ is closed.
	
	Then there exists a regular element $z$ such that $\phi(z)$ does not belong to any  $N$ in $\mathcal N$.
\end{proposition}

\begin{proof} Let $z$ be a regular point. We can therefore assume that $\phi$ is a an embedding on the neighbourhood of $z$ and use freely the identification of a point with its image by $z$. Such a point exists by proposition \ref{pro:reg}. The transversality assumption implies there exists a neighborhood $U$ of $z$ such that $P\defeq\phi^{-1}(\bX)\cap U$ is a real codimension 1 submanifold  of $Z$. Moreover, restricting $U$ further, $P$ consists only of regular points.  For convenience, we  identify $P$ with is image $\phi(P)$ and $U$ with its image $\phi(U)$.

Observe first that for all $N$ in $\mc N$,  $P\cap \partial_\infty N$  is closed in $P$.

We now  prove that for all $N$ in $\mc N$,  $P\cap \partial_\infty N$   has an empty interior in $P$: assume that $P\cap  \partial_\infty N$ has a non empty interior $O_P$ in $P$. It follows that  $U\cap N$  contains the  real submanifold $P\cap N$ of real codimension 1 in $U$. Since both $U$ and $N$ are complex this implies that $U\cap N$, which is a complex analytic set, has codimension zero hence has a non empty interior.  Since $U$ is open in $Z$ and $Z$ is connected, by analytic continuation, if $U\cap N$ is  open in $U$, then $\phi(Z)\subset N$. This is the contradiction. Thus our hypothesis on $Z$ implies that  $P\cap N$  has empty interior in $P$.  

Since $\mc N$ is countable,  it follows by Baire Theorem  that 
$$
P\cap\left(\bigcup_{N\in\mc N} N\right)
$$
has empty interior in $P$. In particular its complementary is  nonempty which is what we wanted to prove.
\end{proof}

\section{Horospherical convergence and curves}
We now restrict ourselves to $(Z,\phi)$, with $Z$ being a curve, satisfying hypothesis $(*)$.
\subsection{Horospherical convergence}
We say a sequence $\seq{y}$ in $Z$ converges {\em horospherically} to the regular point $y$ if the exists a 1-parameter unipotent subgroup $(\uu_t)_{t\in\mathbb R}$ fixing $\phi(y)$ such that 
\begin{enumerate}
	\item $\phi(y_p)$ is a point in $X$ for all $p$,
	\item $\seq{y}$ converges to $y$
	\item $\seqp{\uu_{t_p}(\phi(y_p)}$ converges to a point $y_\infty$ in $X$ for some diverging sequence $\seq{t}$.
\end{enumerate}

\subsection{Good position at a regular point}
Let $y$ be a regular point in $Z$.   
A {\em good position model} is an upper half plane model of $X$ such that   
\begin{enumerate}
	\item $\phi(y)=(0,\ldots,0)\ ,$
	\item $\im (\T_{y}\phi)= 
	\{(z_1,\ldots,z_n)\mid z_2=\ldots=z_n=0\}\eqdef L_0\ .$
	 \end{enumerate}

\begin{lemma}[\sc Good position at a regular point]\label{lem:gp}
After a projective transformation, we can assume that a regular point is in a good position model.
\end{lemma}
\begin{proof}
This follows from lemma \ref{lem:trans}.	
\end{proof}

Obviously,
\begin{lemma}
Let $x_0$ be a point in $L_0\cap X$, then the sequences $\seqt{\uu_t(x_0)}$ in  $L_0$ are converging to $\phi(y)$ both when $t$ goes to $\infty$ and to $-\infty$.	
\end{lemma}

For $y$ a regular point, 
we then  denote also by $\mU=\{\uu_t\}_{t\in\mathbb R}$, the 1-parameter group of unipotent element fixing $\phi(y)$ and $\im(\T_y\phi)$ constructed for proposition \ref{pro:unip} and call it {\em associated to $y$}.

\subsection{Graphs and unipotent}
Let $\phi$ be a holomorphic map of the unit disc in $X$ such that $\phi(0)$ lies in $\bX$ and $\T_0\phi$ is transverse to $\bX$. Up to some change of coordinates and reparametrisation of the source, we can consider  $Z$ to be  an open neighborhood of $0$ in $L_0$ and $\phi$ to be the  the map
$$
z\mapsto(z,z^2f_2(z),\ldots,z^2f_n(z))\ ,
$$
where the $f_i$ are complex analytic.
\begin{proposition}\label{pro:conv0uni}
	Let $B$ be a ball in $L_0\cap X$. Then 
	$$
	\{\uu_{-t}\circ \phi\circ \uu_t\}_{t\in\mathbb R}
	$$
	converges uniformly to the identity  on $B$, both when $t$ goes to $+\infty$ and $t$ goes to $-\infty$  .
\end{proposition}
\begin{proof}
	Indeed let $z$ be a point in $L_0$, since  $$
	\frac{1}{1-it\frac{z}{1+itz}}=1+itz\ .
	$$ then 
\begin{align*}
	&\uu_{-t}\circ \phi\circ \uu_t(z)\\=&\uu_{-t}\left(\frac{z}{1+itz}, \left(\frac{z}{1+itz}\right)^2 f_2\left(\frac{z}{1+itz}\right),\ldots,\left(\frac{z}{1+itz}\right)^2 f_n\left(\frac{z}{1+itz}\right) \right)\\
=&\left(z,\frac{z^2}{1+itz} f_2\left(\frac{z}{1+itz}\right),\ldots, \frac{z^2}{1+itz} f_n\left(\frac{z}{1+itz}\right) \right)\ .
\end{align*}
The result follows since $z\not= 0$ for any $z$ in  $B$.
\end{proof}

\begin{corollary}\label{coro:dd}
Let $B$ be a ball in $L_0\cap X$, then
$$
\lim_{t\to\infty} d(\uu_{t}(z),\phi(\uu_t(z)) =\lim_{t\to-\infty} d(\uu_{t}(z),\phi(\uu_t(z))=0\ ,
$$	
uniformy on $B$ and where $d$ is the complex hyperbolic distance.\end{corollary}
\begin{proof} Since $\uu_t$ is an isometry for the hyperbolic distance 
$$
d(\uu_t(z),\phi(\uu_t(z)) =d(z,\uu_{-t}\phi(\uu_t(z))\ .
$$	
However on a compact set in $X$, the complex hyperbolic distance and the euclidean distance are equivalent.
Thus
$$
0=\lim_{t\to\infty}
\Vert z- \uu_t\phi(\uu_t(z)\Vert=\lim_{t\to\infty}d(z, \uu_t\phi(\uu_t(z))\ \ ,
$$
where the first equality comes from proposition \label{pro:conv0euc} and the second by the equivalence of distances. The same holds when $t$ goes to $-\infty$.
\end{proof}

\section{Proof of the Theorem}
\subsection{The limiting set}
We start with a definition. Let $Z$ be a curve.
\begin{definition}[\sc Limiting set]
Let $y$ be  a regular point, the  {\em limiting set} $W_y(Z)$ is the set of elements $\ell$ in  $\Gamma\backslash \bP(X)$ such that there exists a sequence of points  $\seq{x}$ in $Z$ satisfying  \begin{enumerate}
	\item  the sequence $\seq{x}$ converges horospherically to  $y$,
	\item  the sequence $\seqp{\pi(\phi(x_p),\im\T_{x_p}\pi\phi)}$ converges to $\ell$.
	\end{enumerate}	
\end{definition}

Then as a consequence of proposition \ref{pro:conv0uni}, we have

\begin{proposition}
	\label{pro:lem1} Let $(Z,\phi)$ satisfying hypothesis $(*)$ as in section \ref{sec:hypo*}, $y$ a regular point (identified with $\phi(y)$) and $L$ the complex complex affine hypersurface through $y$, such that $\T_yL=\T_y Z$. Then
$$
W_y(L)\subset W_y(Z)\ .	
$$
\end{proposition}
Actually, one could prove that we have equality but we will not need it.
\begin{proof}
Let $\ell$ be a point in $W_y(L)$. Let then $\seq{x}$ a  sequence in  $L$ converging horospherically to $y$, such that $\seqp{\pi(x_p),\im\T\pi\Phi(x_p)}$ converges to $\ell$. Let $z_p=\phi(x_p)$. We first observe that by continuity of $\phi$, $\seqt{x_p}$ converges to $\phi(y)$. 

Let $\seq{t}$ be a diverging sequence  such that $\seqp{\uu_{t_p}x_p}$ converges to $x_\infty$ in $X$.

By corollary  \ref{coro:dd}, $\seqp{\uu_{t_p}z_p}$ converges to $x_\infty$ as well. Thus $\seq{z}$ converges horospherically to $y$.

Let $z^1_p\defeq(z_p,\T_{z_p}Z)$ and $x^1_p\defeq (x_p,L)$. By proposition \ref{pro:conv0uni}, denoting $d$ again the hyperbolic distance, 
$$
\lim_{p\to\infty}d(z^1_p,x^1_p)=0\ ,
$$
thus 
$$
\lim_{p\to\infty}d(\pi(z^1_p),\ell)=\lim_{p\to\infty}d(\pi(z^1_p),\pi(x^1_p))=0\ .
$$
It follows that $\ell$ belongs to $W_y(Z)$ and 
this concludes the proof.
\end{proof}
Let $\mc N$ be the (countable) set of totally geodesic submanifold of $X$ whose projection on $Y=\Gamma\backslash X$ is dense.
Finally we prove 
\begin{proposition}\label{pro:wldense}
	Assume that $L$ is a complex line  transverse to $\bX$. Assume that $y$ is a  point in  in $\partial_\infty L$ and $y$  does not belong to  $\partial_\infty N$ for any $N$ in $\mc N$.
	Then $W_y(L)$ is dense.
\end{proposition}

\begin{proof} Let $\seqt{\uu_t}$ a 1-parameter group of unipotent in $\mG$ fixing $L$ and $y$. Let $z_0$ be a point in $\cG$ projecting to a point in $L$. Let $N_1$ be the $\seqt{\uu_t}$-orbit of $z_0$, $N_1\defeq \{z_t\}_{t\in\mathbb R}$, where $z_t\defeq \uu_{-t}  z_0$.

By Ratner Theorem \cite[Theorem A]{Ratner:1991aa} there exists a reductive group $\mG_2$ containing $\mU$ in $\mG$ and an orbit $N_2$ of $\mG_2$  containing $N_1$, whose projection on $\Gamma\backslash \cG$ is closed and such that the projection of $N_1$  in $(\Gamma\cap \mG_2)\backslash N_2$ is dense. Observe that $\Gamma\cap \mG_2$ is a  uniform lattice in $\mG_2$.

Let $N$ be the unique totally geodesic subspace produced by proposition \ref{prop:uug2}. In particular $\pi(N)$ is closed since $\Gamma\cap\mG_2$ is a uniform lattice in $\mG_2$.  By proposition \ref{prop:uug2}, $y$ belongs to $\partial_\infty N$. Then, using the hypothesis, $N=X$ and thus $\mG_2=\mG$ and $N_2=\cG$. It follows that $\pi(N_1)$ is dense in $\Gamma\backslash \cG$. Since $\pi(N_1)$ is a (real 1-dimensional) curve, we have that  
$$
V_\alpha\cup V_\omega=\Gamma\backslash \cG\ ,
$$
where $V_\alpha$ and $V_\omega$ are the $\alpha$ and $\omega$-limit set of  $\pi(N_1)$. Since  
$$
W_y(L)\supset V_\alpha\cup V_\omega\ , 
$$
the result follows. 
\end{proof}
\subsection{Conclusion: proof of the main theorem}

We see the complex hyperbolic space $X$ as an open set in  $\bP(E)$.
Our result is that
\begin{theorem}
	Let $\Gamma$ a uniform lattice in the group of biholomorphisms of $X$.
	
	Let $Z$ be a connected complex manifold with boundary, $\phi$ a holomorphic map with values in ${\bf P}(E)$ which is an immersion somewhere. Assume that $\phi(Z)$ intersects $X$ and  $\phi(\partial Z)$ lies in ${\bf P}(E)\setminus X$. 
	
Let $\pi$ be the projection from $X$ to $\Gamma\backslash X$, then  $$\pi(\phi(Z)\cap X)\ ,$$ is dense (for the usual topology) in a totally geodesic subspace of  $\Gamma\backslash X$.
\end{theorem}

\begin{proof} 
We can safely assume that $\phi(Z)$ is not included in any $N$ in $\mc N$. Otherwise, we would work on $N_0$ which is the complex hyperbolic subspace of the lowest possible dimension projecting to a closed totaly geodesic subspace in $\Gamma\backslash X$.

By proposition \ref{pro:regnotin}, there exists a regular element $y$ in $Z$, such that $z\defeq\phi(y)$ does not belong to 
 $$
 \bigcup_{N\in\mc N}\partial_\infty N\ .$$  
 \vskip 0,5 truecm
\noindent{\sc First Step: dimensional reduction.}
Our first step is to reduce to the case  when $Z$ has complex dimension 1. 
 Let $p$ be the dimension of $Z$ and assume $p>1$ for the moment. 
 Let  $\mc F$ the set of complex projective spaces of codimension $p-1$ and intersecting $X$. 
For every $F$ in $\mc F$, let $Z_F\defeq \phi^{-1}(F)$.

There is an non empty open set $\mc F_0$ in $\mc F$, such that for any $F$ in $\mc F_0$, 
\begin{enumerate}
	\item $\phi$ is transverse to $F$,
	\item $\phi(Z_F)$ intersects $X$.
\end{enumerate}
and thus $Z_F$ is a curve with $\partial Z_F\subset \partial Z$. In particular $(Z_F,\phi\vert_{Z_F})$ satisfy hypothesis $(*)$ of definition \ref{sec:hypo*} for all $F$ in $\mc F_0$.   

For every $N$ in $\mc N$, let us consider the following subset of $\mc F_0$.
$$
{\mc F}(N)=\{F\in \mc F_0\mid \phi(Z_F)\subset N\}.
$$ 
Observe now for any $N$ in $\mc N_0$, the subset  ${\mc F}(N)$ is closed: indeed let $\seq{F}$ be a sequence in  ${\mc F}(N)$ converging to $F$ in $\mc F_0$, then for any point $z$ in $Z_F$ there is a sequence of point $\seq{z}$ with $z_p$ in $Z_{F_p}$ converging to $z$. Hence $z$ belongs to $N$, and since this is true for all $z$ in $Z_F$, $Z_F$ is a subset of $N$.

Furthermore assume that   
${\mc F}(N)$ contains an open set $O$, then the following subset of $Z$
$$
U\defeq \bigcup_{F\in O}Z_F 
$$
contains an open set. Since $U\subset N$, by analytic continuation this implies that $\phi(Z)\subset N$ hence that $N=\bP(E)$. 

If follows that by Baire Theorem
$$
\bigcup_{N\in\mc N}{\mc F}(N)\ ,
$$
 is a Baire meagre set. In particular, its complementary $O_{\mc N}$ in $\mc F_0$ is nonmeagre and thus non empty. For $F$ in  $O_{\mc N}$ we have by construction that $(Z_F,\phi\vert_{Z_F})$ is a curve satisfying hypothesis $(*)$ and not included in a 
 $$
\bigcup_{N\in\mc N} N\ .
 $$
It is enough to prove that $\pi(\phi(Z_F))$ is dense. In other words we have reduced to the case of curves. 
\vskip 0,5 truecm
\noindent{\sc Final Step:} We now can assume $Z$ is a curve. Let $z$ be a point in $\bX$ and not in 
$$
\bigcup_{N\in\mc N}\partial_\infty N\ .
$$ 
Let $L$ be any complex line through $z$ transverse to $\bX$. By proposition \ref{pro:wldense}, the limiting set $W_z(L)$ is dense. Take now $z=\phi(y)$ where $y$ is regular in $Z$. Let now $L=\im(\T_y\phi)$ which is transverse at $y$ to $\bX$ since $y$ is regular. Then by proposition \ref{pro:lem1},
$W_z(L)\subset W_y(Z)$. Thus $W_y(Z)$ is dense. We conclude  by noticing that, by construction,  $p(W_y(Z))$ is included in the closure of  $\pi(\phi(Z)\cap X)$.
	\end{proof}

\renewcommand{\MR}[1]{}
\bibliographystyle{amsplain}

\begin{thebibliography}{1}

\bibitem{Bourbaki:owAvyv1m}
Nicolas Bourbaki, \emph{{Lie groups and Lie algebras. Chapters 7--9}}, Elements of Mathematics (Berlin), Springer-Verlag, Berlin, 2005.

\bibitem{Hardt:1983aa}
Robert~M. Hardt, \emph{Some analytic bounds for subanalytic sets}, Differential geometric control theory ({H}oughton, {M}ich., 1982), Progr. Math., vol.~27, Birkh\"auser Boston, Boston, MA, 1983, pp.~259--267. \MR{708507}

\bibitem{Hironaka:1964aa}
Heisuke Hironaka, \emph{Resolution of singularities of an algebraic variety over a field of characteristic zero. {I}}, Annals of Mathematics \textbf{79} (1964), no.~1, 109--203.

\bibitem{Hironaka:1964ab}
\bysame, \emph{Resolution of singularities of an algebraic variety over a field of characteristic zero. {II}}, Annals of Mathematics \textbf{79} (1964), no.~2, 205--326.

\bibitem{Klingler:2016aa}
Bruno Klingler, Emmanuel Ullmo, and Andrei Yafaev, \emph{The hyperbolic {A}x-{L}indemann-{W}eierstrass conjecture}, Publ. Math. Inst. Hautes \'Etudes Sci. \textbf{123} (2016), 333--360. \MR{3502100}

\bibitem{Mok:2019aa}
Ngaiming Mok, Jonathan Pila, and Jacob Tsimerman, \emph{Ax-{S}chanuel for {S}himura varieties}, Ann. of Math. (2) \textbf{189} (2019), no.~3, 945--978. \MR{3961087}

\bibitem{Ratner:1991aa}
Marina Ratner, \emph{Raghunathan's topological conjecture and distributions of unipotent flows}, Duke Mathematical Journal \textbf{63} (1991), no.~1, 235--280. \MR{1106945}

\bibitem{Ullmo:2014aa}
Emmanuel Ullmo and Andrei Yafaev, \emph{Hyperbolic {A}x-{L}indemann theorem in the cocompact case}, Duke Math. J. \textbf{163} (2014), no.~2, 433--463. \MR{3161318}

\end{thebibliography}
\providecommand{\bysame}{\leavevmode\hbox to3em{\hrulefill}\thinspace}
\providecommand{\MR}{\relax\ifhmode\unskip\space\fi MR }
\providecommand{\MRhref}[2]{%
  \href{http://www.ams.org/mathscinet-getitem?mr=#1}{#2}
}
\providecommand{\href}[2]{#2}

\end{document}